\documentclass[a4paper,11pt]{amsart}
\title[The 2nd coeff. of the Alex.\ pol.\ as a satellite obstruction]{The second coefficient of the Alexander polynomial as a satellite obstruction}
\author{Lukas Lewark}
\address{ETH Z\"urich, R\"amistrasse 101, 8092 Z\"urich, Switzerland}
\email{\myemail{llewark@math.ethz.ch}}
\urladdr{\url{https://people.math.ethz.ch/~llewark/}}
\usepackage[latin1]{inputenc}
\usepackage{amsmath,amssymb,stmaryrd,thmtools,enumerate,graphicx,xcolor,enumitem,afterpage,float,amscd,url,mathtools,microtype,tikz-cd,caption,amsthm,calc,nicematrix,multicol,nicefrac}
\usepackage[bookmarks=true,pdfencoding=auto,hidelinks]{hyperref}
\hypersetup{pdfauthor={\authors},pdftitle={\shorttitle}}
\usepackage[nameinlink]{cleveref}
\let\cref\Cref
\crefname{subsection}{section}{sections}
\Crefname{subsection}{Section}{Sections}
\Crefname{enumi}{}{}
\crefname{equation}{}{}
\newcommand{\myemail}[1]{\href{mailto:#1}{#1}}

\newcommand{\qua}{\hskip 0.4em \ignorespaces}
\def\arxiv#1{\relax\ifhmode\unskip\qua\fi
\href{http://arxiv.org/abs/#1}%
{\tt arXiv:\penalty -100\unskip#1}}

\def\MR#1{\relax\ifhmode\unskip\qua\fi
\href{https://mathscinet.ams.org/mathscinet-getitem?mr=#1}{\tt MR#1}}
\def\ZB#1{\relax\ifhmode\unskip\qua\fi
\href{https://zbmath.org/?q=an:#1}{\tt Zbl\:#1}}
\def\xox#1{\csname xx#1\endcsname}

\renewenvironment{thebibliography}[1]{
  \begin{oldthebibliography}{#1}\small
    \setlength{\itemsep}{.5ex}
    \setlength{\parskip}{0em}
}
{
  \end{oldthebibliography}
}
\let\stdthebibliography\thebibliography
\let\stdendthebibliography\endthebibliography
\renewenvironment*{thebibliography}[1]{%
  \stdthebibliography{KWZ19.}}
  {\stdendthebibliography}

\numberwithin{equation}{section}
\declaretheorem{lemma}
\newtheorem{theorem}[lemma]{Theorem}
\newtheorem{corollary}[lemma]{Corollary}
\newtheorem{proposition}[lemma]{Proposition}
\theoremstyle{definition}
\newtheorem{definition}[lemma]{Definition}
\newcommand{\Pos}{\mathcal{P}}
\newcommand{\arc}{h}
\begin{document}
\begin{abstract}
A set $\Pos$ of links is introduced, containing positive braid links as well as arborescent positive Hopf plumbings. It is shown that for links in $\Pos$,
the leading and the second coefficient of the Alexander polynomial have opposite sign.
It follows that certain satellite links, such as $(n,1)$-cables, are not in $\Pos$.
\end{abstract}
\maketitle
In this short note, we introduce a certain set $\Pos$ of fibered links in $S^3$,
which in particular contains non-split positive braid links
and arborescent positive Hopf plumbings, and is closed under connected sum.
We then prove the following.%
\begin{theorem}\label{thm:main}
For all links in $\Pos$,
the second coefficient of the Alexander polynomial is negative.
\end{theorem}
Here, the \emph{Alexander polynomial} $\Delta_L(t)$ is an integer Laurent polynomial in $t^{1/2}$ 
associated with an oriented link~$L\subset S^3$. We use the Conway normalization,
i.e.~$\Delta_U(t) = 1$ for the unknot $U$, and the \emph{skein relation}
\[
\Delta_{L_+}(t) = 
\Delta_{L_-}(t) +
(t^{1/2}  - t^{-1/2}) \Delta_{L_0}
\]
\begin{figure}[h]
\begingroup%
  \makeatletter%
  \providecommand\color[2][]{%
    \errmessage{(Inkscape) Color is used for the text in Inkscape, but the package 'color.sty' is not loaded}%
    \renewcommand\color[2][]{}%
  }%
  \providecommand\transparent[1]{%
    \errmessage{(Inkscape) Transparency is used (non-zero) for the text in Inkscape, but the package 'transparent.sty' is not loaded}%
    \renewcommand\transparent[1]{}%
  }%
  \providecommand\rotatebox[2]{#2}%
  \newcommand*\fsize{\dimexpr\f@size pt\relax}%
  \newcommand*\lineheight[1]{\fontsize{\fsize}{#1\fsize}\selectfont}%
  \ifx\svgwidth\undefined%
    \setlength{\unitlength}{226.31483423bp}%
    \ifx\svgscale\undefined%
      \relax%
    \else%
      \setlength{\unitlength}{\unitlength * \real{\svgscale}}%
    \fi%
  \else%
    \setlength{\unitlength}{\svgwidth}%
  \fi%
  \global\let\svgwidth\undefined%
  \global\let\svgscale\undefined%
  \makeatother%
  \begin{picture}(1,0.2552121)%
    \lineheight{1}%
    \setlength\tabcolsep{0pt}%
    \put(0,0){\includegraphics[width=\unitlength,page=1]{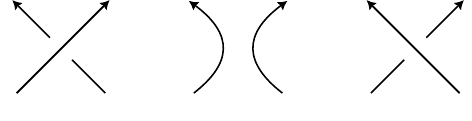}}%
    \put(0.14030253,0.0077326){\color[rgb]{0,0,0}\makebox(0,0)[t]{\lineheight{2.57999992}\smash{\begin{tabular}[t]{c}$L_+$\end{tabular}}}}%
    \put(0.51683279,0.0077326){\color[rgb]{0,0,0}\makebox(0,0)[t]{\lineheight{2.57999992}\smash{\begin{tabular}[t]{c}$L_0$\end{tabular}}}}%
    \put(0.89243512,0.0077326){\color[rgb]{0,0,0}\makebox(0,0)[t]{\lineheight{2.57999992}\smash{\begin{tabular}[t]{c}$L_-$\end{tabular}}}}%
  \end{picture}%
\endgroup%

\caption{The links $L_+, L_0, L_-$ appearing in the skein relation.}
\label{fig:skein}
\end{figure}
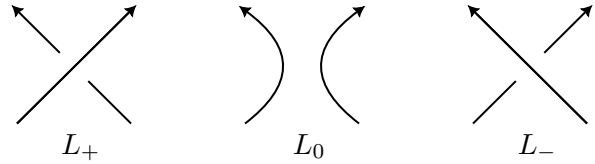

\noindent holds, whenever $L_{\pm}, L_0$ are three links as shown in \cref{fig:skein}.
These equations determine $\Delta_L(t)$ uniquely for all links $L$.
Non-zero Alexander polynomials are of the form
\[
\Delta_L(t) = 
a_d t^d + a_{d-1} t^{d-1} + \ldots + a_{-d+1} t^{-d+1} + a_{-d} t^{-d}
\]
with $d(L) = d \geq 0$ an integer and $a_i = a_{-i}$ if $L$ has an odd number of components,
and $d(L) \geq 1/2$ a half-integer and $a_i = -a_{-i}$ if $L$ has an even number of components.
For $\Delta_L\neq 0$, let us call $\alpha(L) = a_d \neq 0$ the \emph{leading coefficient} and
$\beta(L) = a_{d-1}$ the \emph{second coefficient}.

It is well-known that for fibered links $L$ (such as the links in $\Pos$),
the Alexander polynomial $\Delta_L$ is non-zero,
$d(L)$ equals $b_1(\Sigma)/2$ for $\Sigma\subset S^3$ a fiber surface of $L$,
and $\alpha(L) = \pm 1$.
We shall see in the proof of \cref{thm:main} that $\alpha(L) = 1$ for all $L\in \Pos$.
Thus, if one prefers the convention that the Alexander polynomial is only well-defined up
to multiplication with $\pm t^k$, one may state \cref{thm:main} as follows:
for all links $L$ in $\Pos$, the leading and second coefficient of $\Delta_L$ have opposite sign
(i.e.~$\alpha(L)\cdot \beta(L) < 0$).

\cref{thm:main} and its proof are inspired by Ito's theorem
\cite[Corollary~2]{zbMATH07504315} that $-\beta(L)$
equals the number of prime connected summands of $L$ if $L$ is a positive braid knot.
Our main motivation for \cref{thm:main} is the following obstruction.
Here, we write $P(K)$ for the \emph{satellite link} with pattern $P$ and companion $K$,
for $P$ a link in the solid torus, and $K$ a knot in $S^3$.
\begin{corollary}\label{cor:satellite}
Let $P$ be a pattern with winding number not equal to $\pm 1$ such that
the product of the leading and the second coefficient of $\Delta_{P(U)}$ is non-negative.
Then $P(K)\not\in \Pos$ for all knots $K$.
\end{corollary}
\begin{proof}
The Alexander polynomial of $P(K)$ equals (see e.g.~\cite[Thm.~6.15 and its proof]{zbMATH01092415})
\[
\Delta_{P(K)}(t) = \Delta_K(t^{w(P)}) \cdot \Delta_{P(U)}(t).
\]
It follows that
\begin{align*}
\alpha(P(K)) &= \alpha(K) \cdot \alpha(P(U)),\\
\beta(P(K))  &= \alpha(K) \cdot \beta(P(U)).
\end{align*}
Multiplying the two equations, one sees that the hypothesis $\alpha(P(U))\cdot \beta(P(U)) \geq 0$
implies that $\alpha(P(K))\cdot \beta(P(K)) \geq 0$.
By \cref{thm:main}, it follows that $P(K) \not\in \Pos$.
\end{proof}
The hypotheses for $P$ in \cref{cor:satellite} are e.g.~satisfied by the $(n,1)$-cable patterns for $n\geq 2$, since $P(U)$ is the unknot
(which has $\alpha(U) = 1, \beta(U) = 0$).
So in particular, \cref{cor:satellite} recovers Krishna's recent theorem that $(n,1)$-cables of non-trivial knots are never positive braid knots~\cite[Theorem~1.5]{arXiv:2312.00196}.%

Let us now give the definition of $\Pos$, and prove our results.
\begin{figure}[t]
\begingroup%
  \makeatletter%
  \providecommand\color[2][]{%
    \errmessage{(Inkscape) Color is used for the text in Inkscape, but the package 'color.sty' is not loaded}%
    \renewcommand\color[2][]{}%
  }%
  \providecommand\transparent[1]{%
    \errmessage{(Inkscape) Transparency is used (non-zero) for the text in Inkscape, but the package 'transparent.sty' is not loaded}%
    \renewcommand\transparent[1]{}%
  }%
  \providecommand\rotatebox[2]{#2}%
  \newcommand*\fsize{\dimexpr\f@size pt\relax}%
  \newcommand*\lineheight[1]{\fontsize{\fsize}{#1\fsize}\selectfont}%
  \ifx\svgwidth\undefined%
    \setlength{\unitlength}{281.16027111bp}%
    \ifx\svgscale\undefined%
      \relax%
    \else%
      \setlength{\unitlength}{\unitlength * \real{\svgscale}}%
    \fi%
  \else%
    \setlength{\unitlength}{\svgwidth}%
  \fi%
  \global\let\svgwidth\undefined%
  \global\let\svgscale\undefined%
  \makeatother%
  \begin{picture}(1,0.35352453)%
    \lineheight{1}%
    \setlength\tabcolsep{0pt}%
    \put(0,0){\includegraphics[width=\unitlength,page=1]{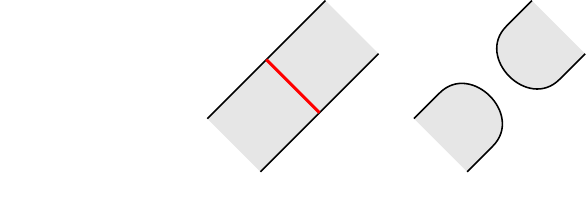}}%
    \put(0.15754369,0.00622423){\color[rgb]{0,0,0}\makebox(0,0)[t]{\lineheight{2.57999992}\smash{\begin{tabular}[t]{c}$L_+$\end{tabular}}}}%
    \put(0.51041217,0.00622423){\color[rgb]{0,0,0}\makebox(0,0)[t]{\lineheight{2.57999992}\smash{\begin{tabular}[t]{c}$L_0$\end{tabular}}}}%
    \put(0.53528678,0.21143403){\color[rgb]{1,0,0}\makebox(0,0)[t]{\lineheight{2.57999992}\smash{\begin{tabular}[t]{c}$h$\end{tabular}}}}%
    \put(0.86328066,0.00622423){\color[rgb]{0,0,0}\makebox(0,0)[t]{\lineheight{2.57999992}\smash{\begin{tabular}[t]{c}$L_-$\end{tabular}}}}%
    \put(0,0){\includegraphics[width=\unitlength,page=2]{surfaces.pdf}}%
  \end{picture}%
\endgroup%

\caption{The links $L_+, L_0, L_-$, the fiber surfaces $\Sigma_+, \Sigma_0$, the arc $h$, and the Seifert surface~$\Sigma_-$ featuring in \cref{def:pos}.}
\label{fig:def}
\end{figure}
\begin{definition}(see~\cref{fig:def})\label{def:pos}
Let $\Pos$ be the smallest set of fibered links in $S^3$ that contains the positive Hopf link
and satisfies the following.
Let $L_0\in\Pos$ with fiber surface $\Sigma_0$, and $\arc\subset\Sigma_0$ a
properly embedded arc. Let $\Sigma_-$ be the surface obtained from $\Sigma_0$ by cutting along $\arc$,
and let $\Sigma_+$ be the surface obtained from $\Sigma_0$ by plumbing a positive Hopf band to $\Sigma_0$ along $\arc$. Let us write $L_{\pm} = \partial \Sigma_{\pm}$.
If the coefficient of $t^{b_1(\Sigma_-)/2}$ in $\Delta_{L_-}(t)$ is less than or equal to $1$
(let us denote this condition by (*)),
then $L_+ \in \Pos$.
\end{definition}
Note that (*) is in particular satisfied if $L_-$ is fibered (because then the coefficient is $\pm 1$ or~$0$),
or if $L_-$ is split (because then $\Delta_{L_-}$ is zero, and so is the coefficient).
Also remark that, as plumbings of positive Hopf bands, all links in $\Pos$ are strongly quasipositive~\cite{zbMATH01097421}. 
\begin{proposition}\mbox{}
\begin{enumerate}
\item $\Pos$ is closed under connected sum.
\item $\Pos$ contains all non-trivial closures of non-split positive braids.
\item $\Pos$ contains all non-trivial arborescent positive Hopf plumbings.
\end{enumerate}
\end{proposition}
\begin{proof}
\textbf{(1)} Let $L, L' \in \Pos$. Let us show that $L \# L' \in \Pos$ (where the connected sum $\#$ is taken along any choice of components), by induction over the first Betti number $b$ of the fiber surface of $L \# L'$. If $b = 2$, then $L$ and $L'$ are both the positive Hopf link, and so~$L \# L' \in \Pos$.
Otherwise, w.l.o.g.~$L_+$ is not the positive Hopf link. Then, by definition of~$\Pos$, one may pick links $L_-$ and $L_0$ for $L_+ = L$ as in \cref{def:pos}. Hence $L_+ \# L'$ is obtained from $L_0 \# L'$ by plumbing a positive Hopf band along an arc $h$, such that cutting along $h$ yields $L_- \# L'$. The link $L_0 \# L'$ is in $\Pos$ by induction. Since $\alpha(L') = 1$ (see proof of \cref{thm:main}) and $L_-$ satisfies (*), so does $L_- \# L'$. It follows that~$L_+ \# L' = L \# L'\in \Pos$.

\textbf{(2)} Let $L$ be a non-trivial non-split link that is the closure of a positive braid in the braid group on $n$ strands, whose standard generators we write as~$\sigma_1, \ldots, \sigma_{n-1}$.
We show by induction over the first Betti number $b$ of the fiber surface of $L$ that $L\in \Pos$.
If $b = 1$, then $L$ is the positive Hopf link, and so $L\in \Pos$.
Let us now assume $b\geq 2$.
An elegant argument by Rudolph~\cite[Pf.~of~last~Prop.]{zbMATH03826789} shows that a positive braid link that is not an unlink can be written as the closure of a positive braid that starts with the square of a generator,
i.e.~is of the form $\sigma_k^2\beta$ with positive~$\beta$.
So let us pick such a braid $\sigma_k^2\beta$ with closure~$L$.
Let $L_0$ and $L_-$ be the links obtained as closures of $\sigma_k \beta$ and $\beta$, respectively.
One may choose an arc $h$ in the fiber surface of $L_0$ such that
plumbing a positive Hopf band along $h$ yields $L_+=L$,
and cutting along $h$ yields $L_-$.
Since $b \geq 2$, it follows that $L_0$ is non-trivial. Moreover, $L_0$ is the closure of the non-split positive braid~$\sigma_k \beta$. Hence $L_0 \in\Pos$ follows by the induction hypothesis.
Finally, $L_-$ is the closure of a positive braid. Thus $L_-$ is either split or fibered.
In either case, $L_-$ satisfies condition~(*). Thus $L_+ = L \in \Pos$ by definition of~$\Pos$.

\textbf{(3)} Let $T$ be a plane tree, and $L$ the associated link,
obtained by placing a positive Hopf band at each vertex of $T$,
and plumbing the bands according to the edges of $T$ (see e.g.~\cite[Ch.~12]{bs},~\cite{MR2165205,zbMATH03938045}).
We proceed by induction over the number of vertices of $T$.
If $T$ consists of a single vertex,
then $L$ is a positive Hopf link, so $L\in\Pos$.
Otherwise, pick a leaf $v$ of $T$ with parent $w$.
Let $L_0$ be the link associated with the plane tree $T\setminus \{v\}$.
There is an arc $h$ in the fiber surface of $L_0$ such that plumbing a positive
Hopf band along $h$ yields the fiber surface of $L_+ = L$, and cutting along $h$
yields a link $L_-$ that is a connected sum of the links associated with the plane tree components of the forest $T\setminus \{v,w\}$. By induction, $L_0 \in \Pos$.
Moreover, $L_-$ satisfies (*) since it is fibered (though it may be not in $\Pos$, namely if $T\setminus \{v,w\}$ is empty and $L_-$ thus the unknot). By definition of $\Pos$, it follows that $L_+ = L\in\Pos$.
\end{proof}
\begin{proof}[Proof of \cref{thm:main}]
Let $L \in \Pos$.
Let us write $b$ for the first Betti number of the fiber surface of $L$, so that $d(L) = b/2$. 
We show $\alpha(L) = a_{b/2}(L) = 1$ and $\beta(L) = a_{b/2 - 1}(L) \leq -1$
by induction over $b$.
If $b = 1$, then $L$ is the positive Hopf link,
with Alexander polynomial $t^{1/2} - t^{-1/2}$
satisfying $\alpha = 1$ and $\beta = -1$.
If $b \geq 2$, by definition there exists a fibered link $L_0 \in \Pos$,
whose fiber surface $\Sigma_0$ contains an arc $\arc\subset\Sigma_0$
such that plumbing a positive Hopf band along $\arc$ yields the fiber surface $\Sigma_+$ of $L_+ = L$,
and cutting along $\arc$ yields a Seifert surface $\Sigma_-$
with $L_- = \partial\Sigma_-$ satisfying~(*).
Since $L_0$ is fibered and $b_1(\Sigma_0) = b - 1$, we have $d(L_0) = b/2 - 1/2$.
Since $L_-$ is the boundary of the Seifert surface $\Sigma_-$ (which may or may not be a fiber surface) with $b_1(\Sigma_-) = b - 2$,
we have $d(L_-) \leq b/2 - 1$. By (*), the coefficient $c$ of $t^{b/2 - 1}$ in $\Delta_{L_-}$
is less than or equal to $1$.
Now, the skein relation 
holds for $L_{\pm}, L_0$ (compare \cref{fig:skein,fig:def}).
Writing $o$ for any linear combinations of powers of $t$ with exponent less than $b/2 - 1$,
we get
\begin{multline*}
\alpha(L_+) t^{b/2} + \beta(L_+) t^{b/2-1} + o = \\
c\cdot t^{b/2-1} + (t^{1/2}  - t^{-1/2})(\alpha(L_0)t^{b/2-1/2} + \beta(L_0)t^{b/2-3/2}) + o.
\end{multline*}
Equating the coefficients of $t^{b/2}$ and $t^{b/2-1}$ gives
\[
\alpha(L_+)   = \alpha(L_0) \qquad\text{and}\qquad
\beta(L_+) = c - \alpha(L_0) + \beta(L_0).
\]
We have $\alpha(L_0) = 1$ and $\beta(L_0) \leq -1$ by induction hypothesis,
and  $c \leq 1$ by~(*).
It follows that $\alpha(L_+) = 1$ and $\beta(L_+) \leq -1$ as desired.
\end{proof}

While I was writing this text, Ito independently
obtained an obstruction similar to \cref{cor:satellite}.
Namely, Theorem~A.2 in \cite{arXiv:2402.01129},
the proof of which likewise relies on the Alexander polynomial,
states that satellites with patterns satisfying a certain condition (different from the condition in \cref{cor:satellite}) are never positive braid knots.%
\medskip

\paragraph{\textbf{Acknowledgments.}}\quad
My gratitude belongs to Siddhi Krishna for discussions about the question
which cables are positive braid knots,
to Filip Misev and Peter Feller for inspiring conversations,
and to Peter Feller also for comments on a first version of the text.
\bibliographystyle{myamsalpha}
\bibliography{References_arXiv}
\end{document}